\documentclass{article}

\usepackage{amssymb}

\usepackage{amsmath}
\usepackage[margin=1in]{geometry}

\usepackage{amsthm}
\newtheorem{remark}{Remark}%
\newtheorem{assumption}{Assumption}
\newtheorem{definition}{Definition}%
\newtheorem{problem}{Problem}
\newtheorem{lemma}{Lemma}
\newtheorem{theorem}{Theorem}

\usepackage{bbold}
\usepackage{xcolor}
\newcommand{\fatone}{\mathbb{1}}

\newcommand{\cov}{\Xi}

\usepackage[ruled,vlined,noend]{algorithm2e}

\usepackage{mathtools}
\mathtoolsset{showonlyrefs=true}

\usepackage{subcaption}
\usepackage{tikz}
\usepackage{pgfplots}
\pgfplotsset{compat=1.18}
\newlength\figH
\newlength\figW
\definecolor{pastel1}{RGB}{255,182,193}  
\definecolor{pastel2}{RGB}{173,216,230}  
\definecolor{pastel3}{RGB}{144,238,144}  
\definecolor{pastel4}{RGB}{255,218,185}  
\definecolor{pastel5}{RGB}{221,160,221}  
\definecolor{pastel6}{RGB}{255,255,224}  
\definecolor{pastel7}{RGB}{176,224,230}  
\definecolor{pastel8}{RGB}{250,128,114}  
\definecolor{pastel9}{RGB}{152,251,152}  
\definecolor{pastel10}{RGB}{255,160,122} %

\usepackage{hyperref}
\hypersetup{hidelinks,breaklinks=true}
\usepackage{url}                 
\usepackage{xurl}
\begin{document}

\title{Differentially Private Formation Control: Privacy and Network Co-Design\thanks{This work was supported by AFOSR under grant FA9550-19-1-0169, by ONR under grant N00014-21-1-2502, and by NSF CAREER grant 1943275. The authors are with the School of Electrical and Computer Engineering, Georgia Institute of Technology, Atlanta, GA, USA.}}

\author{Calvin Hawkins\thanks{Corresponding author. Emails: \{chawkins64,mhale30\}@gatech.edu} \and Matthew Hale}

\date{}

\maketitle

\begin{abstract}
Privacy in multi-agent control is receiving increased attention, though
often a networked system and privacy protections are designed separately, which
can harm performance. 
Therefore, this paper presents a co-design framework for networks and private controllers,
and we apply it to private formation control. Agents' state trajectories are protected using differential privacy, 
and we quantify its impact 
by bounding the steady-state error for private formations. Then, we analyze tradeoffs between privacy level, system performance, and connectedness of the network's communication topology. These tradeoffs are used to formulate a co-design optimization framework to
jointly design agents' communication topology and their privacy levels. Simulation results illustrate the success of this framework.

\end{abstract}

\section{Introduction}\label{sec:intro}
Multi-agent systems require agents to share information to collaborate, and
in some cases this information is sensitive.
For example, self-driving cars may share location data
to be routed to a destination. Location data and other
streams can be revealing and thus 
require protections, though this data must still be useful
for multi-agent coordination. 

This type of privacy has recently been achieved using differential privacy~\cite{dwork2006calibrating}, which has been applied
to various control, estimation, and optimization problems~\cite{geirEnt,yazdani2018differentially,kawano2020design,krishnan2020probabilistic,han2016differentially,huang2015differentially,wang2016differentially}. 
Broadly speaking, differential privacy carefully introduces randomness in a system, e.g., by adding noise to conceal sensitive information.
Differential privacy is appealing because it is immune to post-processing and robust to side information \cite{dwork2014algorithmic}. These properties mean that privacy guarantees are not compromised by performing operations on differentially private data, and that they are not weakened by much by an adversary with additional information about data-producing 
agents~\cite{kasiviswanathan2014semantics}. 

In this work we focus on distributed formation control, e.g., 
robots physically assembling into geometric shapes or non-physical agents maintaining relative state offsets.
We develop a private decentralized formation control protocol in which each agent adds noise to keep its state trajectory private from other agents.
We provide 
privacy guarantees over an agent's entire 
state trajectory, which ensures that an adversary is unlikely to learn much about the state trajectory of a system.
Privacy noise of course impacts agents' formations,
and we provide a closed-form for the steady-state error covariance
that privacy induces in formations. We also bound the steady-state
error itself, and both bounds are in terms of
agents' privacy parameters.

{Privacy is often only a post-hoc concern that is incorporated
after a network and/or a controller is designed, which can harm performance; however, there also exist approaches that achieve privacy without sacrificing convergence~\cite{chen2025differentially}.}
Therefore, we leverage our steady-state analyses
to formulate optimization problems that 
design a network topology and a differential privacy
implementation together.

While other works have implemented other forms of privacy for trajectories
in consensus problems, e.g.,~\cite{zhang2024average},
to the best of our knowledge, we are the first to implement trajectory-level \emph{differential} privacy 
in the setting of formation control. 
The formation control protocol does resemble the average
consensus protocol, and existing work has applied differential privacy to
initial states in consensus, e.g.,~\cite{mitraC,nozari2017differentially,ruan19} and others. This paper differs 
by implementing trajectory-level differential
privacy~\cite{le2013differentially}, which
conceals deviations  from a nominal state trajectory, which
can occur at any time and can occur across multiple timesteps, not just the initial time. 
There also exists work that considers differential privacy and \emph{co-design} in a control setting~\cite{gao2024private}, which considers initial state privacy and uses co-design to tune privacy noise and controller parameters, not the network topology.
There also exists work on the co-design of control and communication, without addressing privacy~\cite{lu2023control}.

This difference in privacy is
necessary because we address a fundamentally
different type of privacy need than initial-state privacy. 
To highlight the difference, suppose that a coalition of drones is assembling
a formation, and, in the course of doing so, one drone takes a detour to photograph
a location of interest, which it would like to conceal. 
This detour occurs after the initial time, and it cannot be concealed by privatizing its initial state.
However, such a detour can
be concealed by implementing privacy at the trajectory level, 
which is what we do.

{Beyond differential privacy, other privacy-preserving approaches have been applied to networked control and estimation.
For example, cryptographic approaches such as homomorphic encryption enable computation over a secure, shared channel 
at the cost of increased computational overhead~\cite{tan2023cooperative}.
Estimation-theoretic approaches based on information bounds have also been used to limit information disclosure by constraining inference accuracy~\cite{guo2025state}.
In this work, we focus on differential privacy because it provides a formal, worst-case notion of privacy that is independent of adversary models and is well suited to decentralized control architectures.}

{
While the privacy mechanism and 
some optimization formalisms employed in this work are standard, the problem formulation and analysis that connect privacy, network topology, and formation control performance are not. Our contributions therefore lie
in making these connections analytically and developing 
a computational framework to apply them in practice. 
}

{
In detail, the main contributions of this paper are as follows:
    \begin{itemize}
        \item An analytical characterization of the steady-state formation error under trajectory-level differential privacy as a function of the network topology
        \item A co-design formulation that jointly selects agent-level privacy parameters and communication weights under steady-state performance constraints
        \item A tractable solution approach based on a biconvex relaxation and alternate convex search which enables practical computation of co-designed networks.
    \end{itemize}
    }

{The main analytical challenge in this work is to bound the steady-state formation error under trajectory-level differential privacy in a manner that explicitly captures the joint dependence on agent-level privacy parameters and the communication topology. This dependence is nonlinear and coupled, and therefore difficult to optimize directly.}
{To overcome this difficulty we derive an explicit upper bound on steady-state error that isolates the dependence on privacy parameters and algebraic connectivity. The resulting bound is structured so that it can be embedded directly into a co-design optimization problem, enabling principled trade-offs between privacy strength, network connectivity, and formation performance.}
    
A preliminary version of this work appeared in \cite{hawkins2020differentially}. 
This paper's new contributions are the co-design framework, the steady-state error analysis, new simulations, and proofs of all
results. 
The rest of the paper is organized as follows. Section~\ref{sec:background} gives background, and Section~\ref{sec:probform} provides problem statements. 
Then Section~\ref{sec:closedform} analyzes private formation control, and 
Section~\ref{sec:codesign} develops the co-design framework. 
Section~\ref{sec:sims} provides simulations, and Section~\ref{sec:conclusions} concludes.

\emph{Notation:} 
$\mathbb{R}$ denotes the reals
and~$\mathbb{N}$ denotes the naturals. 
For~$\ell \in \mathbb{N}$, let~$[\ell] := \{1, \ldots, \ell\}$. 
We use
$I_a\in\mathbb{R}^{a\times a}$ for the identity matrix in $a$ dimensions and $\fatone$ 
for the vector of all ones in $\mathbb{R}^N$. 
We use the norm~$\|v\|_{\infty} = \max_{i \in [N]} |v_i|$.

\section{Background and Preliminaries}\label{sec:background}
Here we review graph theory and differential privacy.

\subsection{Graph Theory Background}
A graph $\mathcal{G} = (V,E)$ is defined over a set of~$N$ nodes 
indexed over $V = \{1, \ldots, N\}$ and a set of edges $E \subseteq V \times V$. The pair $(i,j) \in E$ if nodes $i$ and $j$ share a connection and $(i,j) \notin E$ if they do not. This paper considers undirected, weighted, simple graphs. Undirectedness means that an edge $(i,j) \in E$ is not distinguished from $(j,i) \in E$. Simplicity means that $(i,i) \notin E$ for all $i \in V$. Weightedness means that the edge $(i,j) \in E$ has a weight $w_{ij} = w_{ji} >0$. 

\begin{definition}[Connected Graph]
    A graph $\mathcal{G}$ is connected if, for all $i,j \in \{1,...,N\}$, $i \neq j$, there is a sequence of edges one can traverse from node $i$ to node $j$.
\end{definition}

We assume the following in the rest of the paper. 

\begin{assumption} \label{as:graph}

Agents communicate over a connected, weighted, simple, undirected
graph~$\mathcal{G} = (V, E)$ on~$N$ nodes. 

\end{assumption}

This paper uses the weighted graph Laplacian, which is defined with weighted adjacency and degree matrices. The weighted adjacency matrix $A(\mathcal{G}) \in \mathbb{R}^{N \times N}$ of $\mathcal{G}$ is given by 
\begin{equation*}
    A(\mathcal{G})_{ij} =
    \begin{cases}
    w_{ij} & (i,j) \in E \\
    0 & \text{otherwise}
    \end{cases}.
\end{equation*} 
For undirected~$\mathcal{G}$, $A(\mathcal{G})$ is symmetric.
The degree matrix is~$D(\mathcal{G}) = \textnormal{diag}(d_1,...,d_N)$, and 
the Laplacian is $L(\mathcal{G}) = D(\mathcal{G}) - A(\mathcal{G})$.
Agent $i$'s neighborhood set $N_i$ is the set of all agents 
that agent~$i$ communicates with, i.e., $N_i = \{j \mid (i,j) \in E \}$.
{The weighted degree of 
node $i$ is ${d_i = \sum_{j \in N_i } w_{ij}}.$}

Let $\lambda_k(\cdot)$ be the $k^{th}$ smallest eigenvalue of a symmetric matrix. Then for all graphs $ 0 = \lambda_1(L(\mathcal{G})) \leq \lambda_2(L(\mathcal{G})) \leq \dots \leq \lambda_N(L(\mathcal{G}))$.

\begin{definition}[Algebraic Connectivity \cite{fiedler1973algebraic}]
    The \emph{algebraic connectivity} of a graph $\mathcal{G}$ is the second smallest eigenvalue of~$L(\mathcal{G}),$ and $\mathcal{G}$ is connected if and only if ${\lambda_2(L(\mathcal{G})) > 0}$.
\end{definition}

\subsection{Differential Privacy Background}
This section briefly reviews differential privacy. 
More detail can be found in \cite{le2013differentially,cynthia2006differential}. The goal of differential privacy is to make similar pieces of data appear approximately indistinguishable, and it is immune to post-processing \cite{cynthia2006differential}. 
That is, arbitrary post-hoc computations on private data do not harm differential privacy,
including filtering~\cite{le2013differentially, 9147779}. 

We apply differential privacy to agents' state trajectories, and
we consider vector-valued trajectories of the form
${Z = (Z(1),Z(2),...,Z(k),...),}$
where $Z(k) \in \mathbb{R}^d$ for all $k$. 
The $\ell_2$ norm of $Z$ is defined as $\| Z \|_{\ell_2} = \left(\sum_{k=1}^{\infty} \|Z(k)\|^2_2 \right)^{\frac{1}{2}}$, where $\|\cdot\|_2$ is the ordinary $2$-norm on $\mathbb{R}^d$. 
We privatize trajectories in
$\ell_{\infty}^d = \{ Z \mid \|Z(k)\|_{\infty} < \infty \text{ for all } k\}$, namely the set of vector-valued trajectories whose entries are all finite.
The set~$\ell_{\infty}^d$ contains trajectories that
do not converge, which admits a wide variety of trajectories
seen in control systems.

We consider a network of $N$ agents, where agent $i$'s state trajectory is denoted by $x_i$. Its $k^{th}$ element is denoted $x_i(k) \in \mathbb{R}^d$ for $d \in \mathbb{N}$, and agent $i$'s state trajectory 
belongs to $\ell_{\infty}^{d}$. 
In this work, agents privatize their state trajectories before they are shared, which contrasts with works that privatize collections of trajectories. 
The approach we take is called 
\emph{input perturbation}~\cite{le2013differentially}, and 
it amounts to privatizing data, then using it in some computation. 
This is in place of \emph{output perturbation}, in which one
performs some computation with sensitive data and then privatizes its output. 
In principle, one may envision using output perturbation to privatize
agents' state trajectories. However, the use of output perturbation
requires aggregating agents' states in one place, and such
an aggregator is typically absent in formation control.

Indeed, if an aggregator were available, then all agents
could send their initial states to that aggregator,
it could execute the formation control protocol in a centralized
way, and then the aggregator could 
send each agent their final state to travel to.
The decentralized formation control protocol is thus only
meaningful in the absence of an aggregator, and we therefore
choose a privacy implementation that fits with the lack
of an aggregator, which is input perturbation.

This setup has the advantages of allowing each agent to calibrate
its own privacy parameters, and it also ensures that agents'
information is privatized before it is ever shared. A known challenge
can be that more noise is often required for input perturbation privacy than in alternative
setups.

The goal of differential privacy is to make ``similar'' pieces of data approximately indistinguishable, and an adjacency relation quantifies when pieces of data are ``similar.'' 
For input perturbation, adjacency is defined for single agents. 
 
\begin{definition}[Adjacency]\label{dfn:adj}
Fix an adjacency parameter $b_i > 0$ for agent $i$. 
Then~$\textnormal{Adj}_{b_i}: 
\ell_{\infty}^d \times \ell_{\infty}^d \xrightarrow{} \{0,1\}$ is defined as
\begin{equation*}
    \textnormal{Adj}_{b_i}(v_i,r_i) = \begin{cases}
    1 & \|v_i - r_i\|_{\ell_2} \leq b_i
    \\ 0 & \textnormal{otherwise.}
    \end{cases}
\end{equation*}
\end{definition}
In words, two state trajectories that agent $i$ could produce are adjacent if and only if the $\ell_2$-norm of their difference is upper bounded by $b_i$. This means that every state trajectory within  $\ell_2$ distance $b_i$ from agent $i$'s actual state trajectory must be made approximately indistinguishable from it. 
For example, agent~$i$ may take a detour of size~$b_i$
from the nominal
trajectory that it would follow to get into formation. 
This detour can occur at any time, and the 
adjacency relation in Definition~3
requires that our privacy implementation conceal it, regardless of when it occurs.
Existing approaches to initial-state privacy do not provide
these protections.

While adjacency requires two state trajectories
to have a finite difference between them, it does not
require the two trajectories in question to converge. 
This is reflected in the fact that an agent's trajectories
are contained in the space~$\ell_{\infty}^d$, which
indicates that trajectories consist of real-valued
vectors in~$\mathbb{R}^d$ that can take any finite values, and, 
in particular, they need not ever converge to any value.

Privacy is implemented by a ``mechanism,'' which is a randomized map that
we define next. 
First, fix a probability space $(\Omega, \mathcal{F},\mathbb{P})$. We consider outputs in 
$\ell_{\infty}^{d}$ and use a $\sigma$-algebra over $\ell_{\infty}^{d}$, denoted $\mathcal{F}_{\infty}^{d}$ \cite{hajek2015random}.

\begin{definition}[Differential Privacy]
    \label{dfn:dp}
Let $\epsilon_i > 0$ and $\delta_i \in (0,\frac{1}{2})$ be given.
Fix $b_i>0.$
A mechanism $M: \ell_{\infty}^{d} \times \Omega \xrightarrow{}\ell_{\infty}^{d}$ is $(\epsilon_i,\delta_i)$-differentially private if, for all adjacent $x_i,x_i^\prime \in \ell_{\infty}^{d}$, we have 
    \begin{equation*}
        \mathbb{P}[M(x_i) \in S] \leq e^{\epsilon_i}\mathbb{P}[M(x_i^\prime) \in S] + \delta_i
    \end{equation*}
for all $S \in \mathcal{F}_{\infty}^{d}.$
\end{definition}

{
Definition~\ref{dfn:dp} is applied individually to each agent $i\in[N],$ i.e., each agent's released trajectory satisfies $(\epsilon_i,\delta_i)-$differential privacy with respect to Definition~\ref{dfn:adj}.}
To calibrate privacy, agent $i$ selects privacy parameters $\epsilon_i$ and $\delta_i$. 
Typically, $\epsilon_i \in [0.1, \ln{3}]$ and $\delta_i \in [0, 0.05]$ for all $i$ \cite{le2013differentially,hsu2014differential,kasiviswanathan2014semantics}.
{The value of $\delta_i$ can be regarded as allowing a small probability that the~$\epsilon_i$-differential privacy guarantee for agent $i$ is violated, while $\epsilon_i$ controls the strength of privacy for agent $i$~\cite[Section 2.3]{dwork2014algorithmic}.}

In this work, we use the Gaussian mechanism, which adds
zero-mean i.i.d. noise drawn from a Gaussian distribution. 
We define it using the $Q$-function, i.e., $Q(y) = \frac{1}{\sqrt{2\pi}} \int_y^{\infty} e^{-\frac{z^2}{2}}dz.$
\begin{lemma}[Gaussian Mechanism \cite{le2013differentially}]\label{lem:gausmech}
     Let~$b_i > 0$, $\epsilon_i > 0$, and $\delta_i \in (0,\frac{1}{2})$ be given, and let $x_i \in \ell_{\infty}^{d}$. 
     Then the Gaussian mechanism takes the form ${\tilde{x}_i(k) = x_i(k) + v_i(k)}$, where 
    $v_i$ is a stochastic process with $v_i(k) \sim \mathcal{N}(0, \cov_{v_i})$, where $\cov_{v_i}=\sigma^2_i I_{d}$ with
    $\sigma_i \geq \frac{b_i}{2\epsilon_i}(K_{\delta_i} + \sqrt{K_{\delta_i}^{2} + 2\epsilon_i})$ and~${K_{\delta_i} =  Q^{-1}(\delta_i).}$
    This mechanism provides ($\epsilon_i$,$\delta_i$)-differential privacy to $x_i$.
\end{lemma}

\section{Problem Formulation} \label{sec:probform}

\subsection{Problem Statement}
\begin{problem} \label{prob:main}
Let Assumption~\ref{as:graph} hold. 
Let $x_i(k) \in \mathbb{R}^d$ be agent $i$'s state at time $k$, $N_i$ be agent $i$'s neighborhood set, and $\gamma > 0$ be a stepsize. 
Let~$n_i(k)$ denote the process noise in agent~$i$'s state dynamics at time~$k$. 
We define $\Delta_{ij} \in \mathbb{R}^d$ for all $ (i,j) \in E$ as the desired relative state offset between agents $i$ and $j$. Do each of the following: 
\begin{enumerate}[i.]
    \item Implement the formation control protocol
\begin{equation} \label{nodelevel_no_noise}
    x_i(k+1) = x_i(k) + \gamma \sum_{j \in N_i}w_{ij}(x_j(k) - x_i(k) - \Delta_{ij})+n_i(k),
\end{equation}
in a differentially private, decentralized manner. \label{prob1}
\item Bound the performance of the network in terms of privacy parameters and network properties. \label{prob2}
\item Use those bounds to co-design the communication topology and privacy parameters of the network. \label{prob3}
\end{enumerate}
\end{problem}

Next, we define agent- and network-level dynamics and detail how each agent will enforce differential privacy. 

\subsection{Multi-Agent Formation Control}

The goal of formation control is for agents to reach
a set of relative states. Multi-agent formation
control is a well-researched problem and there are several mathematical
formulations one can use to achieve similar results  \cite{jadbabaie2015scaling,krick2009stabilisation,ren2007information,ren2007consensus,fax2004information,olfati2007consensus,mesbahi2010graph}.
We will define offsets between agents that communicate
and the objective is for all agents to maintain these relative
offsets.

\begin{assumption}[Section~6.1.1 of~\cite{mesbahi2010graph}] \label{as:delta}
The formation specified by~$\{\Delta_{ij}\}_{(i,j) \in E}$ is feasible,
i.e., there exists a set of points that 
simultaneously satisfies all pairwise distances.
We require $\Delta_{ij}=-\Delta_{ji}$
for all $(i,j)\in E.$
\end{assumption}

The objective is having ${\lim\limits_{k\xrightarrow{}\infty}(x_{j}(k)-x_{i}(k))=\Delta_{ij}}$
for all $(i,j)\in E.$ The formation can be centered on any point in
$\mathbb{R}^{d}$, i.e., we allow
formations to be translationally invariant \cite{mesbahi2010graph}.

\begin{remark}
For all~$i \in [N]$, agent~$i$ will execute a privatized
form of~\eqref{nodelevel_no_noise} to be detailed below. 
To ease discussion, we will consider a set of 
points~$\mathcal{P} = \{p_1, \ldots, p_N\}$
below that satisfy~$p_j - p_i = \Delta_{ij}$ for all~$(i, j) \in E$. 
We emphasize here that agents do not know (and do not need to know)
the point~$p_i$ for any~$i \in [N]$. Instead, our use of~$\mathcal{P}$
is only for ease of discussion and analysis. Algorithm~\ref{alg:forms} below
shows the steps that each agent will actually implement,
and none of them require~$p_i$. 
\end{remark}

As in~\eqref{nodelevel_no_noise}, 
we model agents as discrete-time single integrators, i.e., $x_i(k+1)=x_i(k)+u_i(k)+n_i(k),$
where $n_i(k)\sim\mathcal{N}(0,s_i^2I_d)$ is process noise and $u_i(k)\in\mathbb{R}^d$ is agent $i$'s input. 
Let the vector $p=(p_{1}^{T},\dots,p_{N}^{T})^{T}\in\mathbb{R}^{Nd}$ 
be a network-level state that satisfies the formation specification, 
let $x(k)=(x_{1}(k)^{T},...,x_{N}(k)^{T})^{T}\in\mathbb{R}^{Nd}$,
and let $n(k)=(n_{1}(k)^{T},...,n_{N}(k)^{T})^{T}\in\mathbb{R}^{Nd}$.
Set~$\bar{x}(k) = x(k) - p$. Then analyzing~\eqref{nodelevel_no_noise}
at the network-level is equivalent to analyzing
\begin{equation} \label{networklevel_nonoise}
\bar{x}(k+1)=\Big((I_N-\gamma L\big(\mathcal{G})\big)\otimes I_{d}\Big)\bar{x}(k) +n(k), 
\end{equation}
where~$\otimes$ is the Kronecker product; we reiterate
that~\eqref{networklevel_nonoise} is used only for analysis, and
agents do not need to know~$p_i$ for any~$i$.

Let $\bar{x}_{i_{[l]}}$ be the $l^{th}$ scalar element of $\bar{x}_{i}$ for~$l \in [d]$. 
Then $\bar{x}_{[l]} = [\bar{x}_{1_{[l]}},\dots,\bar{x}_{N_{[l]}}]^{T}\label{statediml}$
is the vector of all agents' states in the $l^{th}$ dimension, and 
$\bar{n}_{[l]} = [n_{1_{[l]}},\dots,n_{N_{[l]}}]^{T}$
is the vector of corresponding noise terms. 
Analyzing~\eqref{networklevel_nonoise} is therefore equivalent to analyzing the protocol 
\begin{equation}
\bar{x}_{[l]}(k+1)=(I_N-\gamma L(\mathcal{G}))\bar{x}_{[l]}(k)+n_{[l]}(k)\label{dimJproto}
\end{equation}
for all $l \in [d]$ simultaneously.

\subsection{Private Communications (Solution to Problem~\ref{prob:main}.\ref{prob1})}
To privately implement the protocol in~\eqref{nodelevel_no_noise}, 
agent $j$ starts by selecting privacy parameters $\epsilon_{j}>0$,
$\delta_{j}\in(0,\frac{1}{2})$, and adjacency relation $\text{Adj}_{b_{j}}$
with $b_{j}>0$. Then, agent $j$ privatizes its state
trajectory $x_{j}$ with the Gaussian mechanism. Let $\tilde{x}_{j}$
denote the differentially private version of $x_{j}$, where $\tilde{x}_{j}(k)=x_{j}(k)+v_{j}(k),$ with $v_{j}(k)\sim\mathcal{N}(0,\cov_{v_{j}}).$
Here, $\cov_{v_{j}}=\sigma_{j}^{2}I_{d}$ and 
$\sigma_{j} \geq \frac{b_j}{2\epsilon_j}\big(K_{\delta_j} + \sqrt{K_{\delta_j}^{2} + 2\epsilon_j}\big)$.
Lemma~\ref{lem:gausmech} shows that this setup keeps agent~$j$'s state
trajectory~$(\epsilon_j, \delta_j)$-differentially private. 
Agent $j$ then shares $\tilde{x}_{j}(k)$ with its neighbors.

Thus, agent $i$
only has access to $\tilde{x}_{j}(k)$ for $j\in N_i$. Plugging
this into the node-level formation control protocol in \eqref{nodelevel_no_noise}
gives
\begin{equation}
\tilde{x}_{i}(k+1)=\tilde{x}_{i}(k)+\gamma\sum_{j\in N_i}w_{ij}\big(\tilde{x}_{j}(k)-\tilde{x}_{i}(k) - \Delta_{ij}\big)+n_i(k).\label{DP_node_level}
\end{equation}
It is well known that~\eqref{dimJproto} 
and~\eqref{DP_node_level} are mathematically
equivalent (see, e.g.,~\cite{mesbahi2010graph}). 
We analyze~\eqref{dimJproto} because it is simpler,
and all of these analyses apply verbatim to~\eqref{DP_node_level}. 

We present
the full private formation controller in
Algorithm~\ref{alg:forms}, which solves
Problem~\ref{prob:main}.\ref{prob1}.

\begin{algorithm}
\caption{Private formation controller (Solution to Problem~\ref{prob:main}.\ref{prob1})}
    \label{alg:forms}
    \SetAlgoLined
\textbf{Inputs}: $\epsilon_i > 0$, $\delta_i \in (0, 1/2)$, $b_i > 0$ for all $i \in [N]$, $\Delta_{ij}$ for all~$(i, j) \in E$ \\
    \textbf{Outputs}: Privately assembled formation \\
    \For{$k = 1, 2, \ldots$}{
    \For{$i \in [N]$}{
    Agent~$i$ sends $\tilde{x}_i(k)$ to agents~$j \in N_i$ \\
    Agent~$i$ receives~$\tilde{x}_j(k)$ for all~$j \in N_i$ \\
    Agent~$i$ executes the state update in~\eqref{DP_node_level}
    }
    }
\end{algorithm}
{We note that the Gaussian input perturbation mechanism employed here is standard. The novelty of this work lies not in developing a new mechanism but instead
in the system-level analysis of privacy's impact on multi-agent control performance,
as well as the incorporation of this analysis into a 
computational framework for privacy/network co-design, which we develop across
this section and Sections~4 and~5.}
To solve Problems~\ref{prob:main}.\ref{prob2} and~\ref{prob:main}.\ref{prob3}, 
we will analyze this protocol
at the network level.

\begin{lemma} \label{lem:ntwrk_dyn}
Consider Algorithm~\ref{alg:forms} and
let Assumption~\ref{as:graph} hold. 
Then 
the network-level dynamics for each $l \in\{1,\dots,d\}$ are
\begin{equation}
\bar{x}_{[l]}(k+1)=(I_N-\gamma L(\mathcal{G}))\bar{x}_{[l]}(k)+z_{[l]}(k), 
\end{equation}
 where $z_{[l]}(k)\sim\mathcal{N}(0,\cov_{z}),$ ${\cov_{z}=\gamma^{2}A(\mathcal{G})\cov_{v}A(\mathcal{G})+\cov_n}$, $\cov_v=\textnormal{diag}(\sigma_1^2,\dots\sigma_N^2),$ and $\cov_n=\textnormal{diag}(s_1^2,\dots,s_N^2)$.

\end{lemma}
\begin{proof}
    See~\ref{sec:ntwrk_dyn_proof}. 
\end{proof}

To analyze network-level performance, we define
$\beta_{[l]}(k):=\frac{1}{N}\fatone^{T}\bar{x}_{[l]}(k)\fatone,$
which is the state vector that the protocol in \eqref{dimJproto}
would converge to with initial state $x_{[l]}(k)$ and without privacy or process noise.
Also let 
\begin{equation}
e_{[l]}(k)=\bar{x}_{[l]}(k)-\beta_{[l]}(k),\label{errordiml}
\end{equation}
which is the offset of the current state from the state the protocol
would converge to without noise.

\section{Performance of Differentially Private Formation Control} \label{sec:closedform}
In this section we solve Problem \ref{prob:main}.\ref{prob2}.
First, we 
compute the total mean square error in a formation. 
Then, we use these results to derive performance bounds.
In the forthcoming proofs, we 
will analyze each dimension separately for simplicity.
However, results will be stated
in terms of all dimensions simultaneously because it is that form
of network-level performance that is ultimately of interest.

We quantify steady-state performance using the total mean-square error
of the network, denoted~$e_{ss}$. 
Agent $i$'s private formation control protocol in
Algorithm~\ref{alg:forms} runs 
in~$\mathbb{R}^d$, and each agent in the network runs this protocol. This is equivalent to 
running $d$ identical
copies of \eqref{dimJproto}, which is in $\mathbb{R}^N$. Thus, using~\eqref{dimJproto}, we can
compute the mean square error in dimension $l$ and then multiply
by $d$ to compute $e_{ss}$. 
The mean-square error in dimension~$l$ is equal
to~$\lim_{k \to \infty} \frac{1}{N} \sum_{i=1}^{N} E[e^2_{[l],i}(k)],$ where $e_{[l],i}(k)$ is the $i^{th}$ element of $e_{[l]}(k)$ in~\eqref{errordiml}.  
Then we have $
e_{ss}:=\lim_{k\to\infty}\frac{d}{N}\sum_{i=1}^{N}E\left[e_{[l],i}^{2}(k)\right].$

\subsection{Connections with the Lyapunov Equation}
\label{sec:lyap}
The main error bound in this paper uses the fact that we can represent the
total error in the system as the trace of a covariance matrix. We define
$\cov_{e_{[l]}}(k)=E\left[e_{[l]}(k)e_{[l]}(k)^{T}\right]$ and $\cov_{\infty}=\lim_{k\to\infty}\cov_{e_{[l]}}(k)$;
we omit the subscript~$l$ since~$\cov_{\infty}$ is the same
in each dimension. 

Then we have $e_{ss}=\frac{d}{N}Tr(\cov_{\infty}).$ Now we will
analyze the dynamics of $e_{[l]}(k)$ and $\cov_{e_{[l]}}(k)$. 
For a given~$\gamma > 0$ and a given graph~$\mathcal{G}$, 
let~${\mathcal{M}=I_N-\gamma L(\mathcal{G})-\frac{1}{N}\fatone\fatone^T}.$ 
Then we have the following.
\begin{lemma}
\label{lem:jbberror}
Consider Algorithm~\ref{alg:forms}, and let
Assumptions~\ref{as:graph} and~\ref{as:delta} hold. 
Then the network level error $e_{[l]}(k)$ in~\eqref{errordiml}
evolves via 
\begin{equation}
e_{[l]}(k+1)=\mathcal{M}e_{[l]}(k)+(I_N-\frac{1}{N}\fatone\fatone^T)z_{[l]}(k),\label{error_dyn}
\end{equation}
 $\cov_{e_{[l]}}(k)$ evolves according to 
\begin{equation}
\cov_{e_{[l]}}(k+1)=\mathcal{M}\cov_{e_{[l]}}(k)\mathcal{M}+(I_N-\frac{1}{N}\fatone\fatone^T)\cov_{z}(I_N-\frac{1}{N}\fatone\fatone^T),\label{cov_dyn}
\end{equation}
 and $\cov_{e_{[l]}}(k)$ can be computed via
\begin{equation}
\cov_{e_{[l]}}(k)=\sum_{i=0}^{k-1}\mathcal{M}^{i}(I_N-\frac{1}{N}\fatone\fatone^T)\cov_{z}(I_N-\frac{1}{N}\fatone\fatone^T)\mathcal{M}^{i}.
\label{cov_expand}
\end{equation}
\end{lemma}

\begin{proof}
    See~\ref{sec:error_dyn_proof}.
\end{proof}
The following result computes the spectral radius of a matrix
that will be used in our analysis below.
A similar bound is 
in~\cite{you2011network}, 
but that work selects 
a vector-valued gain, while we only vary~$\gamma$, which is a scalar-valued stepsize.

\begin{lemma} \label{lem:M_properties}
Let~$N$ agents communicate over a graph that 
satisfies Assumption~\ref{as:graph} 
with Laplacian $L(\mathcal{G})$.
Then the maximum singular value  of ${\mathcal{M}=I_N-\gamma L(\mathcal{G})-\frac{1}{N}\fatone\fatone^T}$
is $\sigma_{\max}(\mathcal{M})=1-\gamma\lambda_{2}(L)$. 

\end{lemma}
\begin{proof}
    See~\ref{sec:M_proof}.
\end{proof}
We prove the existence of the steady-state error covariance matrix~$\cov_{\infty}$ by showing
that it solves a discrete-time Lyapunov equation that has a unique solution. 

\begin{theorem}
\label{thm:lyap}

Consider~$N$ agents executing Algorithm~\ref{alg:forms}, and 
let Assumptions~\ref{as:graph} and~\ref{as:delta} hold. 
Then $\cov_{\infty}$ is the unique solution to the discrete time Lyapunov equation
\begin{equation}
 \cov_{\infty}=(I_N-\frac{1}{N}\fatone\fatone^T)\cov_z(I_N-\frac{1}{N}\fatone\fatone^T)
 +\mathcal{M}\cov_\infty \mathcal{M},  \label{lyap_eq} 
\end{equation}
where we have $\cov_z=\!\gamma^{2}A(\mathcal{G})\textnormal{diag}(\sigma_1^2,\dots,\sigma_N^2)A(\mathcal{G})
 \!+\!\cov_n$ and~${\sigma_i^2 = 
\frac{b_i}{2\epsilon_i}(K_{\delta_i} + \sqrt{K_{\delta_i}^{2} + 2\epsilon_i})}$ 
is the variance of privacy noise.

\end{theorem}
\begin{proof}
    From Lemma~\ref{lem:jbberror}, we
have 
$
\cov_{\infty} = \sum_{i=0}^{\infty}\mathcal{M}^{i}(I_N-\frac{1}{N}\fatone\fatone^T)\cov_{z}(I_N-\frac{1}{N}\fatone\fatone^T)\mathcal{M}^{i}.
$
Let $Q=(I_N-\frac{1}{N}\fatone\fatone^T)\cov_z(I_N-\frac{1}{N}\fatone\fatone^T).$
Using~$\mathcal{M}^0 = I_N$ and factoring gives
$
\cov_{\infty} = Q+\mathcal{M}\left(\sum_{i=1}^{\infty}\mathcal{M}^{i-1}Q\mathcal{M}^{i-1}\right)\mathcal{M} 
  = Q + \mathcal{M}\cov_{\infty}\mathcal{M}.
$
By Lemma~\ref{lem:M_properties}, 
$\mathcal{M}$ has eigenvalues strictly less than $1$, and by Prop.~2.1 in~\cite{gahinet1990sensitivity}, \eqref{lyap_eq} has a unique,
symmetric solution.
\end{proof}

Theorem~\ref{thm:lyap}, along with the fact that $e_{ss}=\frac{d}{N}Tr(\cov_{\infty})$,
allows us to solve for $e_{ss}$. Thus, given a graph $\mathcal{G}$ and privacy parameters $\{(\epsilon_i,\delta_i)\}_{i=1}^N,$ we can determine the performance of the network, encoded by $e_{ss}$ before runtime. 
We will use its value to design a communication topology that allows agents to be as private as possible while satisfying~$e_{ss}\leq e_R$, where $e_R$ is the maximum allowable error at steady-state. 

\subsection{Analytical Bound}
\label{sec:error_bounds}

Next we derive a scalar bound on $e_{ss}$ in terms of the privacy parameters $\{(\epsilon_i,\delta_i)\}_{i=1}^N,$ and properties of the graph $\mathcal{G}$. 

\begin{theorem}
\label{thm:error_bound}

Let all the conditions from Theorem~\ref{thm:lyap} hold. With $e_{ss}:=\lim_{k\to\infty}\frac{d}{N}\sum_{i=1}^{N}E\left[e_{[l],i}^{2}(k)\right],$
we have
\begin{equation}
e_{ss}\!\leq\!\frac{\gamma d \sum_{i=1}^{N}\!{C_{w_i,d_i}}
{\frac{b_i}{2\epsilon_i}\Big(K_{\delta_i} \!+\! \sqrt{K_{\delta_i}^{2} \!+\! 2\epsilon_i}\Big)} \!+\! \frac{N-1}{N\gamma}\sum_{i=1}^N\! s_i^2}{N\lambda_{2}(L(\mathcal{G}))\left(2-\gamma\lambda_{2}(L(\mathcal{G}))\right)}, \label{ethm2}
\end{equation}
where
$C_{w_i,d_i} = \sum_{j=1}^{N}w_{ij}^{2} -\frac{d_i^2}{N}$. 

\end{theorem}
\begin{proof}
    See~\ref{sec:big_proof}.
\end{proof}
Theorem~\ref{thm:error_bound} solves Problem~\ref{prob:main}.\ref{prob2}, and
we will use it below to design networks for both privacy and performance.   
For example, if the network must achieve $e_{ss}\leq e_R$, then
this can be achieved by requiring that
the bound in~\eqref{ethm2} is less than $e_R$. 
This requirement will impose joint constraints on the communication topology $\mathcal{G},$ the weights in the communication topology $w_{ij},$ and privacy parameters $\{(\epsilon_i,\delta_i)\}_{i=1}^N$, which we will use in co-design.

\section{Privacy and Network Co-design}
\label{sec:codesign}
In this section we solve Problem~\ref{prob:main}.\ref{prob3}. 
Our goal
is to design a network that meets global performance requirements and agent-level privacy requirements. 
If some agents
use strong privacy, 
then the high-variance privacy noise they use may make
global performance poor, even if
many other agents have only weak privacy requirements. This effect
can also be influenced by the topology
of the network, e.g., if an agent with strong privacy sends very noisy
messages to other agents along heavily weighted edges. 
Privacy/network co-design thus requires unifying and balancing these tradeoffs.

This co-design framework is designed for agents
that wish to participate in a joint co-design setup. 
Agent~$i$ is able to constrain the value of~$\epsilon_i$
that is generated by the co-design problem, and this 
constraint is permitted with the understanding that
agent~$i$ will use the value of~$\epsilon_i$
that is computed in solving the co-design problem.

\subsection{Co-Design Search Space}
We consider
$N$ agents that implement private formation control and each
agent has a minimum strength of privacy that it will accept. 
Our objective is to make agents as private as possible while meeting a global performance requirement.
In many settings,
each agent's neighborhood set will be fixed \emph{a priori} based on hardware
compatibility or physical location, and these neighborhoods specify an unweighted, 
undirected graph of edges that can be used. 

We will design the privacy parameters for each agent and all
edge weights for the given graph. 
That is, we can assign a zero or non-zero weight to each edge that is present,
but we cannot assign non-zero weight to an edge that is absent. 
We denote the given unweighted graph of possible edges by~$\mathcal{G}_0$ and denote
its unweighted Laplacian by~$L_0$. 
We define~$\mathcal{L}(L_{0})$ as the space of all
weighted graph Laplacians~$L$ such that~$L_{ij} = 0$
if~$L_{0,ij} = 0$. 
Co-design will find the optimal weighted
Laplacian in~$\mathcal{L}(L_0)$ and privacy
parameters for each agent. 

\subsection{Co-Design Problem Statement}
Below we derive the privacy/network co-design problem. 

\begin{problem}[Ideal Co-design]
\label{prob:codesign}
Given an undirected, unweighted, simple graph $\mathcal{G}_{0}$ with Laplacian~$L_0,$ a required global error bound $e_R,$ an edge weight budget~$B,$ and a minimum level of privacy $\epsilon_i^{max}$ for agent $i$ for all $i\in[N]$, to co-design privacy and agents' communication topology, solve
\begin{align}
&\min_{L(\mathcal{G})\in\mathcal{L}(L_{0}),\{\epsilon_i\}_{i=1}^N} \quad \sum_{i=1}^N\epsilon_{i}^{2}-\lambda_{2}(L(\mathcal{G})) \label{eq:budget-biconvex-codesign}\\
&\textnormal{subject to }\\
&\hspace{-9pt}\frac{\gamma d \sum_{i=1}^{N}{C_{w_i,d_i}}
{\frac{b_i}{2\epsilon_i}\Big(K_{\delta_i} \!+\! \sqrt{K_{\delta_i}^{2} \!+\! 2\epsilon_i}\Big)} \!+\! \frac{N-1}{N\gamma}\sum_{i=1}^N s_i^2}{N\lambda_{2}(L)\left(2-\gamma\lambda_{2}(L)\right)} \!\leq\! e_{R}\\
 \quad&\quad\epsilon_{i}\leq\epsilon_{i}^{max}\quad\textnormal{for all }i\textnormal{ and }\textnormal{Tr}(L(\mathcal{G}))\leq 2B
.\end{align} 
\end{problem}

{In the objective function, the term $\sum_{i = 1}^{N} \epsilon_{i}^{2}$, helps make each agent's privacy stronger since
privacy is stronger as~$\epsilon_i$ shrinks. 
{We choose a quadratic function to penalize large privacy parameters more heavily,
but this choice is not essential for our framework, and any convex
function of the decision variables~$\{\epsilon_i\}_{i \in [N]}$ is admissible.
}
For the term $-\lambda_{2}(L(\mathcal{G})),$ minimizing this helps produce a network which is as connected as possible, i.e., maximizing the algebraic connectivity will produce a solution which promotes better network-level convergence properties.
} 
{The first constraint requires that
the total mean square error of the network  is less than
some user-defined value $e_{R}$. 
The second constraint
requires that each agent's privacy level obeys the constraint
that the agent has specified. 
The last constraint budgets how much edge weight can be used, specifically the total edge weight used may not exceed~$B.$} 

Overall, solving Problem~\ref{prob:codesign} will produce a communication topology and privacy parameters that will meet a global performance constraint while allowing agents to be as private as possible. This solves Problem~\ref{prob:main}.\ref{prob3}.
The main computational challenge in solving Problem~\ref{prob:codesign}
comes from the objective function and first constraint. Both depend on the algebraic
connectivity of the graph~$\mathcal{G}$, namely~$\lambda_2\big(L(\mathcal{G})\big)$, and for~$N$ nodes 
the computation of~$\lambda_2\big(L(\mathcal{G})\big)$
has~$O(N^3)$ complexity~\cite{parlett2000qr}. 
Furthermore, Problem~\ref{prob:codesign} is nonconvex due to the constraint
on steady-state error, and thus one typically cannot expect 
to find a global solution to it.
To address these challenges while maintaining the essential structure of the co-design problem, we develop a relaxed formulation in the following subsection.

\subsection{Biconvex Co-Design}

We introduce a biconvex relaxation of Problem~\ref{prob:codesign} that has more favorable numerical properties while maintaining the structure of the original problem. We present the relaxed problem first, then discuss its derivation, biconvexity, and solution approach.

\begin{problem}[Implemented Co-design]
\label{prob:biconvex-codesign}
Given the same inputs as Problem~\ref{prob:codesign}, solve

\begin{align}
&\min_{L(\mathcal{G})\in\mathcal{L}(L_{0}),\{\epsilon_i\}_{i=1}^N,y\in[0,N]} \quad \sum_{i=1}^N\epsilon_{i}^{2}-y \label{eq:biconvex_cod}\\
&\textnormal{subject to }\\
&\hspace{-9pt}\frac{\gamma d \sum_{i=1}^{N}{C_{w_i,d_i}}
{\frac{b_i}{2\epsilon_i}\Big(K_{\delta_i} \!+\! \sqrt{K_{\delta_i}^{2} \!+\! 2\epsilon_i}\Big)} \!+\! \frac{N-1}{N\gamma}\sum_{i=1}^N s_i^2}{Ny\left(2-\gamma y\right)} \!\leq\! e_{R}\\
 \quad&\quad\epsilon_{i}\leq\epsilon_{i}^{max}\quad\textnormal{for all }i\textnormal{ and }\textnormal{Tr}(L(\mathcal{G}))\leq 2B \\
 \quad & \quad y\leq \lambda_2(L(\mathcal{G})).
\end{align} 
\end{problem}

To derive Problem~\ref{prob:biconvex-codesign}, we decouple the nonconvex constraint involving $\lambda_2(L(\mathcal{G}))$ by introducing an auxiliary variable $y\in[0,N]$.
We rewrite the error constraint from Problem~\ref{prob:codesign} as
\begin{equation}
    \frac{\gamma d}{N e_R} \sum_{i=1}^{N}{C_{w_i,d_i}}
{\frac{b_i}{2\epsilon_i}\Big(K_{\delta_i} \!+\! \sqrt{K_{\delta_i}^{2} \!+\! 2\epsilon_i}\Big)} \!+\! \frac{N-1}{e_R N^2\gamma}\sum_{i=1}^N s_i^2 \!\leq\! \lambda_{2}(L)(2-\gamma \lambda_{2}(L(\mathcal{G}))),
\end{equation}
and then split this into two separate constraints: first requiring that the left-hand side be bounded by $y(2-\gamma y)$, and second requiring that $y\leq\lambda_2(L)$.
The problem still requires careful handling due to the product of $C_{w_i,d_i}$ and the function of $\epsilon_i$ in the first constraint. 
This product is what creates the biconvex structure that we analyze next.

{\begin{definition}[Biconvex program~\cite{floudas1990global}]
\label{def:bi}
Let $X \subseteq \mathbb{R}^m$ and $Y \subseteq \mathbb{R}^n$ be convex sets, and let $(x,y) \in X \times Y$. 
A function $f : X \times Y \to \mathbb{R}$ is biconvex if, for every fixed $y \in Y$, the function $x \mapsto f(x,y)$ is convex on $X$, and for every fixed $x \in X$, the function $y \mapsto f(x,y)$ is convex on $Y$. 
A constraint set $S \subseteq X \times Y$ is biconvex if, for every fixed $y \in Y$, the set $\{x \in X \mid (x,y) \in S\}$ is convex, and for every fixed $x \in X$, the set $\{y \in Y \mid (x,y) \in S\}$ is convex. 
An optimization problem of the form
$$
\min_{(x,y)\in X\times Y} f(x,y) \quad \text{subject to } (x,y)\in S
$$
is called a biconvex program if $f$ and $S$ are biconvex.
\end{definition}}

This framework requires partitioning the decision variables of Problem~\ref{prob:biconvex-codesign} into two disjoint sets, and we partition them into $(L(\mathcal{G}_0),y)$ and $\left(\{\epsilon_i\}_{i=1}^N\right).$
The following theorem establishes that Problem~\ref{prob:biconvex-codesign} is biconvex.
{\begin{theorem}\label{thm:codesign-biconvex}
    Given an undirected, simple, weighted graph $\mathcal{G}_{0},$ a global error requirement $e_R,$ an edge budget $B,$ and a minimum level of privacy $\epsilon_i^{max}$ for agent $i$ for all $i\in[N],$
    Problem~\ref{prob:biconvex-codesign} is a biconvex program with respect the variables $(L(\mathcal{G}_0),y)$ and $\{\epsilon_i\}_{i=1}^N$.
\end{theorem}}
\begin{proof}
     The objective function~\eqref{eq:biconvex_cod} is positive definite and quadratic as a function of $\{\epsilon_i\}_{i=1}^N$ and is thus convex in the privacy parameters.
      The objective is linear in $y.$ This implies that the objective is biconvex in $(L(\mathcal{G}_0),y)$ and $\{\epsilon_i\}_{i=1}^N.$
    
     {To analyze the error constraint
\begin{equation}\label{eq:error_constraint_biconvex}
    \frac{\gamma d}{N e_R} \sum_{i=1}^{N}{C_{w_i,d_i}}
    {\frac{b_i}{2\epsilon_i}\Big(K_{\delta_i} \!+\! \sqrt{K_{\delta_i}^{2} \!+\! 2\epsilon_i}\Big)} +\! \frac{N-1}{e_R N^2}\sum_{i=1}^N s_i^2 -y(2-\gamma y) \!\leq\! 0,
    \end{equation}
     recall that $C_{w_i,d_i} = \sum_{j=1}^{N}w_{ij}^{2} -\frac{d_i^2}{N}$ and ${d_i = \sum_{j \mid (i,j) \in E} w_{ij}}.$ For $i\in[N],$ let $w_i=[w_{i1},\dots,w_{iN}]^T\in\mathbb{R}^N.$ Then
     \begin{align}
         C_{w_i,d_i} &= \sum_{j=1}^{N}w_{ij}^{2} -\frac{d_i^2}{N}=w_i^T w_i-\frac{1}{N}\big( \fatone^Tw_i\big)^2\\
         &=w_i^T w_i-\frac{1}{N}w_i^T\fatone\fatone^Tw_i=w_i^T\big( I_N-\frac{1}{N}\fatone \fatone^T\big)w_i.
     \end{align}
    Note that $I_N-\frac{1}{N}\fatone \fatone^T\succeq0$  and thus $C_{w_i,d_i}$ is convex for all $i\in[N].$
    It can be shown that $\frac{b_i}{2\epsilon_i}\big(K_{\delta_i} +\sqrt{K_{\delta_i}^{2} \!+\! 2\epsilon_i}\big)$ is a convex, nonincreasing function of $\epsilon_i$ for any $b_i>0$ and~$\delta_i\in(0,1/2).$}
    Lastly, $-y(2-\gamma y)$ is convex in $y.$
    Therefore, for fixed $y$ and $w_{ij},$ the error constraint~\eqref{eq:error_constraint_biconvex} is convex in $\{\epsilon_i\}_{i=1}^N.$
    Similarly, for fixed $\{\epsilon_i\}_{i=1}^N,$ the error constraint~\eqref{eq:error_constraint_biconvex} is convex in $y$ and $w_{ij}.$
    Thus~\eqref{eq:error_constraint_biconvex} is biconvex in
    $(L(\mathcal{G}_0),y)$ and $\{\epsilon_i\}_{i=1}^N.$

    The algebraic connectivity $\lambda_{2}(L(\mathcal{G}))$ can be expressed as $
         \lambda_2(L(\mathcal{G}))=\min_{x\textnormal{ s.t. }\|x\|=1,\ x^T\fatone=0} x^T L(\mathcal{G}) x.$
    This is an infimum over a set of linear maps and is therefore concave as a function of edge weights, and the constraint $y\leq\lambda_{2}(L(\mathcal{G}))$ is convex.
    Lastly, the constraint $\textnormal{Tr}(L(\mathcal{G}))\leq 2B$ is a linear function of edge weights and is therefore convex.
    Overall, this shows that Problem~\ref{prob:biconvex-codesign} meets the definition of a biconvex program from {Definition~\ref{def:bi}}.
\end{proof}

Having established that Problem~3 is biconvex, we now consider approaches for its solution.
{A global optimization algorithm for biconvex programs exists \cite{floudas1990global}, but it guarantees global optimality by iteratively partitioning the feasible region and maintaining associated convex relaxations and bounds for each subregion of the partition. 
As the number of subregions grows, this method must store a growing collection of subproblems and bounding information, which can lead to exponential space complexity in the worst case. Instead, we employ an alternate convex search approach \cite[Algorithm 4.1]{gorski2007biconvex} that maintains a single iterate and alternates between two convex subproblems. At each iteration, only the current decision variables and fixed problem data are stored, and thus the space complexity scales as $O(|E| + N)$.}
Specifically, our algorithm alternates between: (1) solving for optimal edge weights $\{w_{ij}\}_{i,j\in[N]}$ and auxiliary variable $y$ while fixing privacy parameters $\{\epsilon_i\}_{i\in[N]}$, and (2) solving for optimal privacy parameters $\{\epsilon_i\}_{i\in[N]}$ while fixing weights and $y$.

{This approach has several desirable properties. First, each subproblem is convex and can be solved efficiently using standard convex optimization tools. 
The wide availability of software for solving optimization problems therefore makes it simple to deploy this framework in practice. 
Second, the alternating scheme is guaranteed to converge to a stationary point since the objective function is bounded below and either decreases or remains constant at each iteration~\cite{gorski2007biconvex}. 
We deliberately avoid adding any more assumptions that could strengthen our convergence analysis because we want this co-design framework
to have the broadest reach possible. 
Third, while global optimality is not analytically guaranteed, our numerical results in the following section demonstrate that our solutions 
yield high-quality co-designs that successfully meet both privacy and formation requirements in practice.
In fact, as we will show, computational co-design results often exceed the minimum privacy requirements, exceed the minimum performance requirements,
or exceed both, indicating strong levels of performance in practice. 
The usefulness of these stationary points follows from the analytical performance characterizations developed earlier, which ensure that any feasible stationary solution satisfies the desired privacy and steady-state performance guarantees.
Strong performance numerically thus follows from the preceding analytical developments that have been made. 
}
\section{Simulations}
\label{sec:sims}
In this section we provide simulation results for optimal privacy/network 
co-design\footnote{{All code is available on GitHub~\cite{hawkingithub}.}}.
{This section does not provide any numerical comparisons to prior work
because we are not aware of any comparable prior work that co-designs
a network topology and a trajectory-level differential privacy implementation.}
We first define an input topology as the undirected, unweighted, simple graph $\mathcal{G}$ shown in Figure~\ref{FixedTopology}. Then, we run privacy/network co-design 
to obtain edge weights and privacy parameters~$\{\epsilon_i\}_{i=1}^N$.
In these results, the smaller a node is drawn the more private it is, i.e., $\epsilon_i$ gets smaller as node $i$ shrinks, and the thicker an edge is drawn the more weight it has. In simulation, edges with weights less than~$10^{-4}$ are considered deleted.
{
Within each example, the same solution was reached regardless of the initial iterate used by the optimization algorithm. 
}

\emph{Example 1: (Trading off privacy and performance)}
Fix the input graph $\mathcal{G}$ shown in Figure~\ref{FixedTopology}. Fix 
$\epsilon^{max}= [0.4 ,0.9, 0.55, 0.35, 0.8, 0.45, 0.7, 0.5, 0.52, 0.58]^T,$
$B=6,$ $\gamma=1/2N,$ $\delta_i=0.05,$ and $b_i=1$ for all $i.$ Now let $e_R$ take on the values $e_R\in\{8,16,64\}.$ For each of these values, privacy/network co-design was used to design the weighted graphs shown in Figure~\ref{fig:er_tune} and the privacy parameters shown in Figure~\ref{fig:eR_tune_bar}.

In Figure~\ref{fig:er_tune}, as we allow weaker performance, quantified by larger $e_R,$ the agents are able to use a much stronger level of privacy. This is illustrated by the nodes shrinking as $e_R$ increases from Figure~\ref{fig:biggest_er} to~\ref{fig:smallest_er}.
The edge weights are roughly constant due to the fixed budget of $B=6$ and the biconvex nature of the problem.
\begin{figure}
\centering
    \includegraphics[width=0.15\textwidth]{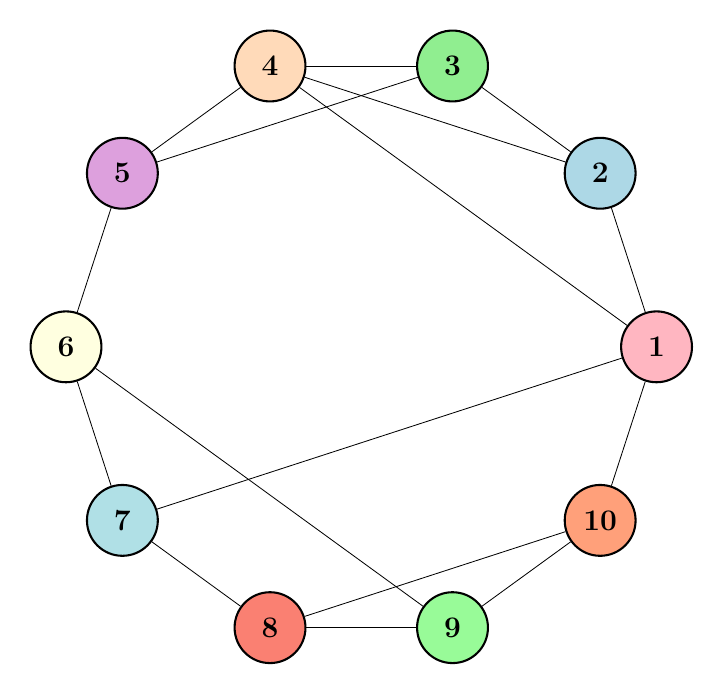}
    \caption{The graph used for simulations. Edge weights are not specified and will be optimized over.
    }
    \label{FixedTopology}
\end{figure}
In Figure~\ref{fig:eR_tune_bar}, we also see that as the required level of performance decreases, co-design allows the agents to be more private. This trend persists for all agents in varying magnitudes, which are influenced by the topology and each $\epsilon_i^{max}.$ 
Moreover, Figure~\ref{fig:eR_tune_bar} shows that, even under
the tightest constraint on~$e_{ss}$, agents implement values of~$\epsilon_i$ ranging from approximately~$0.04$ to~$0.07$, which
are in line with some of the strongest privacy protections that have been used in the literature~\cite{hsu2014differential}. 
Thus, even when using input perturbation for privacy, 
there is an inherent compatibility between privacy and formation control, and
agents can simultaneously attain both high performance and strong privacy protections. 
\hfill$\triangle$

\begin{figure*}
\begin{subfigure}[t]{\textwidth}
    \centering
           \subfloat[$e_R= 8$\label{fig:biggest_er}]{
      \scalebox{.3}{\includegraphics[]{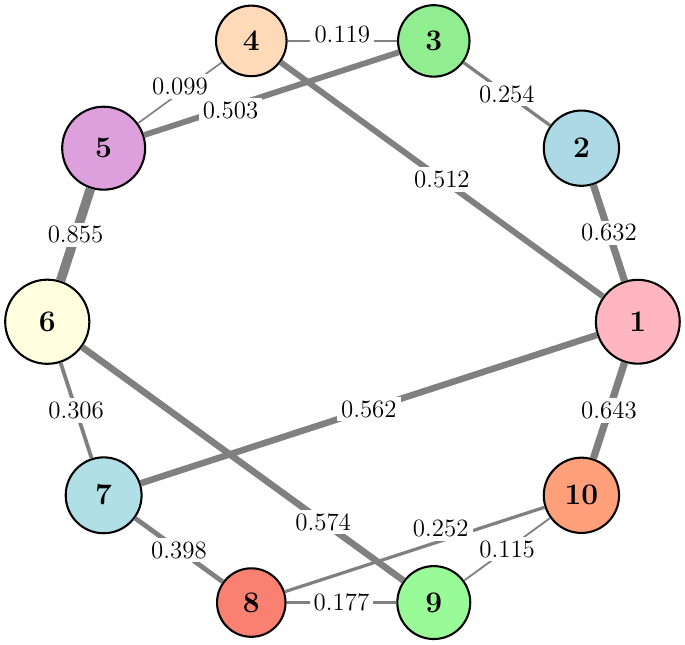}}
      }
      \centering
       \subfloat[$e_R=16$]{%
        \scalebox{.3}{\includegraphics[]{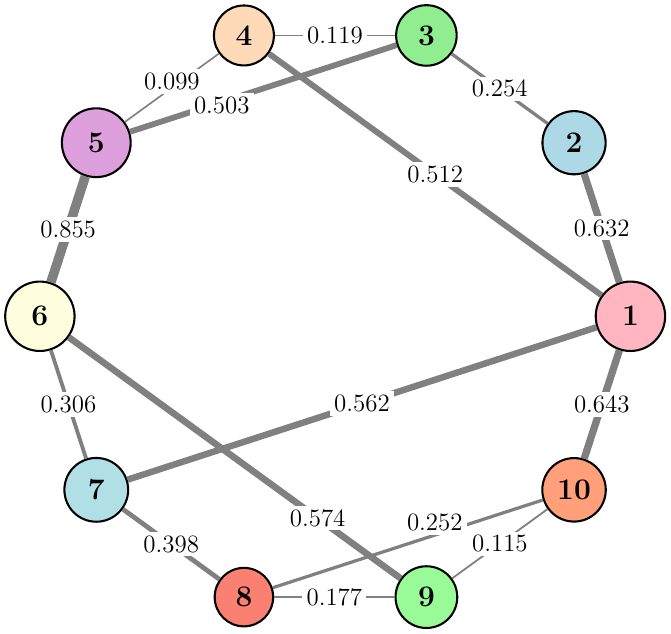}}
        }
   \subfloat[$e_R=64$\label{fig:smallest_er}]{
   \scalebox{.3}{\includegraphics[]{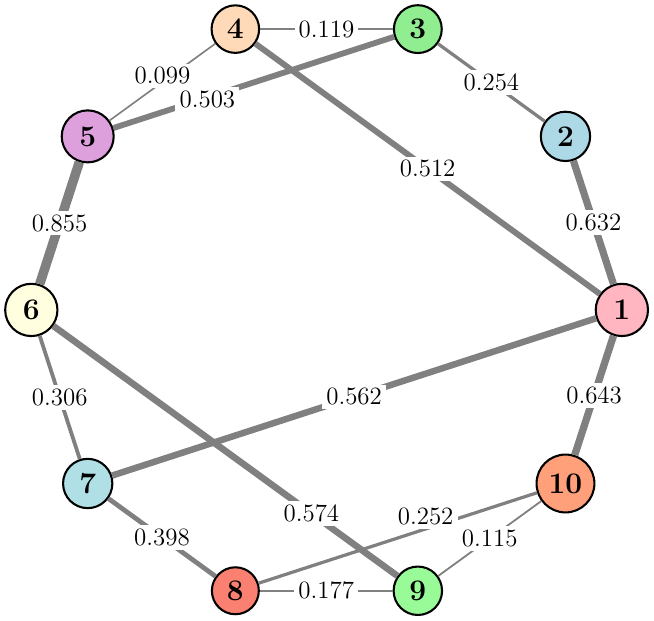}}
   }
    \end{subfigure}
  \caption{{Graphs generated in Example $1$ for $e_R\in\{8,16,64\}.$ The smaller a node is drawn, the smaller
  its value of $\epsilon_i$ is. The thicker an edge is drawn, the larger its weight is. When we allow weaker performance, indicated by larger $e_R,$ each agent has a stronger level of privacy and the topology uses less weight as illustrated by the nodes shrinking from Figure~\ref{fig:biggest_er} to Figure~\ref{fig:smallest_er}.}
}
  \label{fig:er_tune} 
\end{figure*}
\begin{figure}
\vspace{-\baselineskip}
    \centering
    \setlength{\figH}{5cm}
    \setlength{\figW}{9cm}
    \includegraphics[]{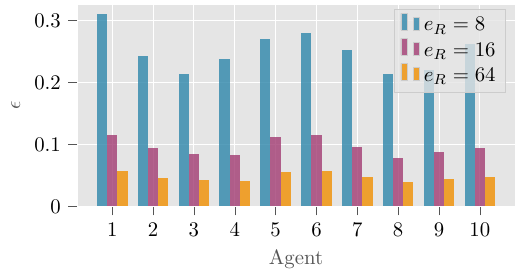}
    \caption{{The privacy parameters designed in Example $1$ with $e_R\in\{8,16,64\}$. }}
    \label{fig:eR_tune_bar}
\end{figure}

\emph{Example 2: (Trading off edge budget and privacy)}
We fix all the parameters as given in Example $1,$ however, we set $e_R=16.0$ and let $B$ take on the values $B\in\{2,4,8\},$ i.e., we fix a level of performance and modulate the edge weight budget.
The co-designed edge weights and privacy parameters are given in Figures~\ref{fig:tune_budget} and~\ref{fig:budget_bars}.

In Figure~\ref{fig:tune_budget}, as the budget is increased more edge weight is used and the agents get \emph{less} private. As $B$ is increased more edge weight is used to increase the $\lambda_2(L(\mathcal{G}))$ term of the objective. However, high edge weights amplify the privacy noise injected by agents, which has a negative impact on steady state error. Therefore, to meet the performance requirement of $e_R=16,$ co-design is making agents less private when we produce networks with high edge weights. This is directly controlled by $B.$ The numerical values of  the privacy parameters are presented in Figure~\ref{fig:budget_bars}, where we observe the privacy parameters roughly doubling between $B=2$ and $B=8.$
{\begin{remark}
    These results show that the proposed privacy and network co-design framework effectively handles the increased amount of noise required by input perturbation by incorporating the privacy noise in the system analysis and design.
\end{remark}}

\begin{figure*}
    \centering
  \subfloat[$B=2$\label{fig:smallest_lambda}]{%
      \scalebox{.3}{\includegraphics[]{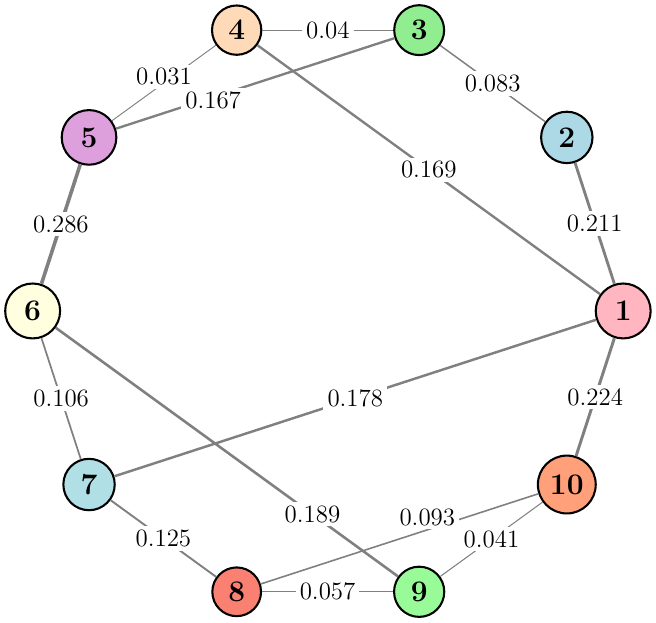}}
      }
  \subfloat[$B=4$]{%
        \scalebox{.3}{\includegraphics[]{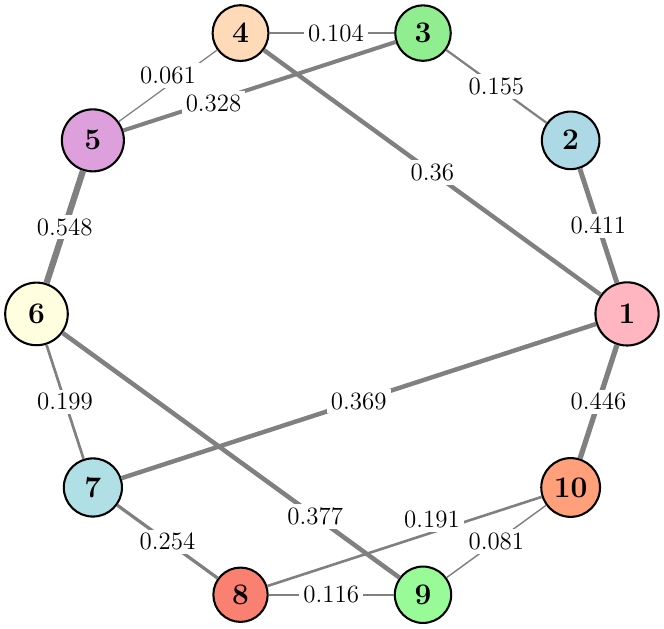}}
        }
  \subfloat[$B=8$\label{fig:biggest_lambda}]{%
        \scalebox{.3}{\includegraphics[]{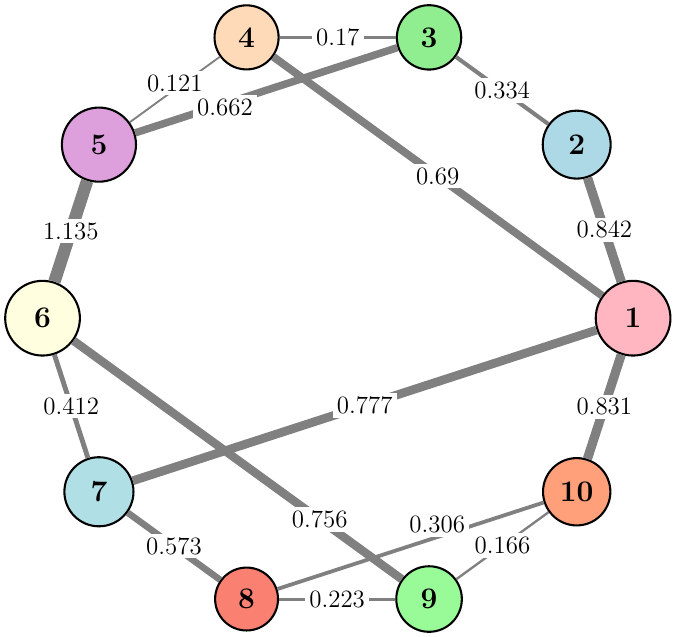}}
        }
  
  \caption{ The outputs of privacy/network co-design with fixed parameters specified in Example $2$ for different values of $B\in\{2,4,8\}.$}
  \label{fig:tune_budget} 
\end{figure*}

\begin{figure}
    \centering
    \setlength{\figH}{5cm}
    \setlength{\figW}{9cm}
    \includegraphics[]{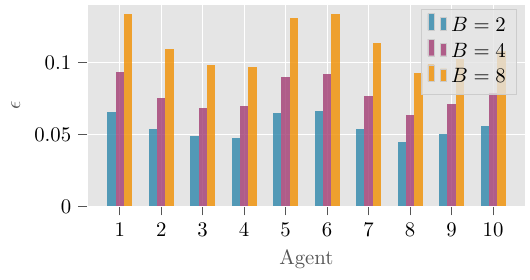}
    \caption{The privacy parameters from Example $2$ with $B\in\{2,4,8\}$. }
    \label{fig:budget_bars}
\end{figure}
\section{Conclusions}\label{sec:conclusions}
We have developed a framework to design networks and 
privacy parameters for differentially private formation control,
which allows agents to privately assemble formations with bounded
steady-state error. This framework was developed for cooperative agents,
and future work includes quantifying the extent to which it is robust
to non-cooperative and adversarial agents. Such agents may inject
false information into the network (among other attacks) and a promising
direction for future work is to explore how 
privacy's insensitivity to changes in one agent's state trajectory
can help make such attacks less effective. 

\appendix
\section{Proofs of Supporting Lemmas}
\subsection{Proof of Lemma~\ref{lem:ntwrk_dyn}}\label{secA1}
\label{sec:ntwrk_dyn_proof}
From~\eqref{DP_node_level}, we subtract $p_i$ from both sides to find
\[
\bar{x}_{i}(k+1)=\bar{x}_{i}(k)+\gamma\sum_{j\in N_i}w_{ij}(\bar{x}_{j}(k)+v_{j}(k)-\bar{x}_{i}(k))+n_i(k),
\]
 which we factor as 
\begin{equation}
\bar{x}_{i}(k+1)=\bar{x}_{i}(k)+\gamma\sum_{j\in N_i}w_{ij}(\bar{x}_{j}(k)-\bar{x}_{i}(k))+z_i(k),
\end{equation}
where $z_{i}(k)=\gamma\sum_{j\in N_i}w_{ij}v_{j}(k)+n_i(k).$ Note that 
$z_{i}(k)=\gamma \big([A(\mathcal{G})]_{row\ i}\otimes I_d \big)v(k) +n_i(k),
$
where~$[A(\mathcal{G})]_{row\ i} \in \mathbb{R}^{1 \times N}$ is the~$i^{th}$ row
of~$A(\mathcal{G})$ and~$\otimes$ is the Kronecker product. 
 Next we define the vector~$$z_{[l]}(k)=[z_{1_{[l]}}(k)^{T},\dots,z_{N_{[l]}}(k)^{T}]^{T}\in\mathbb{R}^N$$ 
giving $z_{[l]}(k)=\gamma A(\mathcal{G})v_{[l]}(k)+n_{[l]}(k)$ at the network level. The covariance of $z_{[l]}(k)$ can then be calculated as
${\cov_{z}  = E[\left(\gamma A(\mathcal{G})v_{[l]}(k)+n_{[l]}(k)\right)\left(\gamma A(\mathcal{G})v_{[l]}(k)+n_{[l]}(k)\right)^{T}]}$.
Since $v_{[l]}(k)$ and $n_{[l]}(k)$ are statistically independent, $E[v_{[l]}(k)n_{[l]}(k)^T]=E[n_{[l]}(k)v_{[l]}(k)^T]=0$. Applying 
this and the linearity of expectation gives
${\cov_{z} 
 =\gamma^{2}A(\mathcal{G})\cov_{v}A(\mathcal{G})+\cov_n}$.
 Then $z_{[l]}(k)\sim\mathcal{N}(0,\gamma^{2}A(\mathcal{G})\cov_{v}A(\mathcal{G})+\cov_n).$
 \hfill $\qed$

\subsection{Proof of Lemma~\ref{lem:jbberror}}
\label{sec:error_dyn_proof}
Expanding the error term~$e_{[l]}$ gives

$e_{[l]}(k+1) =\left(I_N-\gamma L(\mathcal{G})-\frac{1}{N}\mathbb{1}\mathbb{1}^{T}\right)\bar{x}_{[l]}(k)+\left(I_N-\frac{1}{N}\mathbb{1}\mathbb{1}^{T}\right)z_{[l]}(k),$
where we have used the fact that $L(\mathcal{G})\fatone=0.$ Let ${\mathcal{M}=I_N-\gamma L(\mathcal{G})-\frac{1}{N}\fatone\fatone^T}$ and note that we also have
$
\mathcal{M}\beta_{[l]}(k)=\left(I_N-\frac{1}{N}\mathbb{1}\mathbb{1}^{T}\right)\beta_{[l]}(k)=0.
$
Then we have the equality
$e_{[l]}(k+1) 
 =\mathcal{M}e_{[l]}(k)+\left(I_N-\frac{1}{N}\mathbb{1}\mathbb{1}^{T}\right)z_{[l]}(k)$,
which proves~\eqref{error_dyn}.
Plugging this into the definition of $\cov_{e_{[l]}}(k+1)$ gives
$
\cov_{e_{[l]}}(k+1) =\mathcal{M}\cov_{e_{[l]}}(k)\mathcal{M}+\left(I_N-\frac{1}{N}\mathbb{1}\mathbb{1}^{T}\right)\cov_{z}\left(I_N-\frac{1}{N}\mathbb{1}\mathbb{1}^{T}\right),
$
which proves~\eqref{cov_dyn}. Then~\eqref{cov_expand} follows by applying~\eqref{cov_dyn} recursively.\hfill$\qed$

\subsection{Proof of Lemma~\ref{lem:M_properties}}
\label{sec:M_proof}
We begin by analyzing the eigenvalues of $\mathcal{M}.$ Note that $\mathcal{M}\fatone=0,$
and thus $\mathcal{M}$ has eigenvalue $0$ with eigenvector $\fatone.$ Now let $(\lambda,\ v)$ be an eigenpair of $\mathcal{M}$ with $\lambda\neq0.$ We will show that $\lambda\neq0$ implies that $(\lambda,\ v)$ is
also an eigenpair of $I-\gamma L(\mathcal{G}).$ Since $\mathcal{G}$ is connected and undirected, $I-\gamma L(\mathcal{G})$ is a doubly stochastic matrix.
Thus, it has one eigenvalue of modulus $1$ with an eigenvector of $\fatone$, and the other eigenvalues are positive and strictly inside the unit disk \cite[Lemma 3]{olfati2007consensus}.

Using $\fatone^{T}(I_N-\gamma L(\mathcal{G}))=\fatone^{T},$ we have $
\fatone^{T}v=\fatone^{T}(I_N-\gamma L(\mathcal{G}))v.$
Adding and subtracting $\fatone^{T}(\frac{1}{N}\fatone\fatone^T)v$ gives
$\fatone^{T}v =\fatone^{T}\mathcal{M}v+\fatone^{T}(\frac{1}{N}\fatone\fatone^T)v$,
then plugging in $\mathcal{M}v=\lambda v$ and $\fatone^{T}(\frac{1}{N}\fatone\fatone^T)=\fatone^{T}$ gives $
\fatone^{T}v =\lambda\fatone^{T}v+\fatone^{T}v=(\lambda+1)\fatone^{T}v.
$
Since $\lambda\neq0,$ we must have $\fatone^{T}v=0$, and thus $v$
is orthogonal to $\fatone.$ Furthermore we have that $(\frac{1}{N}\fatone\fatone^T)v=0,$ so
$\lambda v=\mathcal{M}v =(I_N-\gamma L(\mathcal{G})-\frac{1}{N}\fatone\fatone^T)v=(I_N-\gamma L(\mathcal{G}))v$. 
 Any non-zero eigenvalue of $\mathcal{M}$ is also an eigenvalue of $I-\gamma L(\mathcal{G})$ and the eigenvector $v$ is orthogonal to $\fatone.$

Furthermore, $L(\mathcal{G})$ and $I_N-\gamma L(\mathcal{G})$ have the same eigenvectors and $\lambda_i(I_N-\gamma L(\mathcal{G}))=1-\gamma\lambda_i(L(\mathcal{G})).$ Thus, $\lambda_i(\mathcal{M})=1-\gamma\lambda_i(L(\mathcal{G}))$ for ${i\in\{2,\dots,N\}}$, and
we have $\lambda_{\max}(\mathcal{M})=1-\gamma \lambda_2(L(\mathcal{G}))<1.$
Thus, all eigenvalues of $\mathcal{M}$ lie strictly in the unit disk. Now $\sigma_{\max}(\mathcal{M}) = 1-\gamma\lambda_{2}(L)$
follows from the fact that
$\left[\sigma_{\max}(\mathcal{M})\right]^{2} =\left[\left(\lambda_{\max}\left((\mathcal{M})(\mathcal{M})\right)\right)^{1/2}\right]^{2}=\lambda_{\max}\left((\mathcal{M})\right)^{2}=\left(1-\gamma\lambda_{2}(L(\mathcal{G}))\right)^{2}.$
\hfill$\qed$

\section{Proof of Theorem~\ref{thm:error_bound}}
\label{sec:big_proof}
With Theorem~\ref{thm:lyap} and~\cite[Equation (151)]{kwon1996bounds} we have 
\begin{equation}
Tr\left(\cov_{\infty}\right)\leq\frac{Tr(Q)}{1-\left(\sigma_{\max}(\mathcal{M})\right)^{2}},\label{survey_bound}
\end{equation}
where $\sigma_{\max}(\mathcal{M})$ denotes the maximum singular value of
$\mathcal{M}.$

We next expand $Tr(Q)$ in~\eqref{survey_bound}
and apply the cyclic permutation property of the trace to find 
\begin{align}
 Tr(Q) &=Tr\left((I_N-\frac{1}{N}\fatone\fatone^{T})(I_N-\frac{1}{N}\fatone\fatone^{T})\cov_z\right).\label{trace_temp}
 \end{align}
Note that $(I_N-\frac{1}{N}\fatone\fatone^{T})(I_N-\frac{1}{N}\fatone\fatone^{T})=I_N-\frac{1}{N}\fatone\fatone^{T}.$ Plugging this and $\cov_z=\gamma^{2}A(G)\cov_{v}A(G)+\cov_{n}$ into~\eqref{trace_temp} gives 
\begin{equation}
    Tr(Q) =\gamma^{2}Tr\left((I_N-\frac{1}{N}\fatone\fatone^{T})A(G)\cov_{v}A(G)\right)+Tr\left((I_N-\frac{1}{N}\fatone\fatone^{T})\cov_{n}\right).\label{total_trace}
\end{equation}
With $\cov_{n}=\textnormal{diag}(s_{1}^{2},\dots,s_{N}^{2})$, the $i^{th}$ diagonal
term of $(I_N-\frac{1}{N}\fatone\fatone^{T})\cov_{n}$ is given by $
\left[(I_N-\frac{1}{N}\fatone\fatone^{T})\cov_{n}\right]_{ii}=(1-\frac{1}{N})s_{i}^{2}.$
Then we have that the last term in~\eqref{total_trace} is
\begin{equation}
Tr\left((I_N-\frac{1}{N}\fatone\fatone^{T})\cov_{n}\right)=\frac{N-1}{N}\sum_{i=1}^N s_{i}^{2}. \label{first_trace}
\end{equation}
To simplify $\gamma^{2}Tr\left((I_N-\frac{1}{N}\fatone\fatone^{T})A(G)\cov_{v}A(G)\right)$,
cyclic permutation of the trace gives that
$$
    \gamma^{2}Tr\left((I_N-\frac{1}{N}\fatone\fatone^{T})A(G)\cov_{v}A(G)\right)=\gamma^{2}Tr\left(A(G)(I_N-\frac{1}{N}\fatone\fatone^{T})A(G)\cov_{v}\right).
$$
The $i^{th}$ diagonal term of $A(G)(I_N-\frac{1}{N}\fatone\fatone^{T})A(G)$ is $\left[A(G)(I_N-\frac{1}{N}\fatone\fatone^{T})A(G)\right]_{ii} =\sum_{j=1}^N w_{ij}^{2}-\frac{d_{i}^{2}}{N},$
and
\begin{equation}
\gamma^{2}Tr\left(A(G)(I_N-\frac{1}{N}\fatone\fatone^{T})A(G)\cov_{v}\right) = \gamma^{2}\sum_{i=1}^N C_{w_i,d_i}\sigma_{i}^{2}.\label{second_trace}
\end{equation}
Using~\eqref{first_trace}-\eqref{second_trace} in~\eqref{total_trace} gives
$
Tr(Q)=\gamma^{2}\sum_{i=1}^N C_{w_i,d_i}\sigma_{i}^{2}+\frac{N-1}{N}\sum_{i=1}^N s_{i}^{2},
$
and plugging this into~\eqref{survey_bound} gives
$Tr\left(\cov_{\infty}\right)\leq\frac{\gamma^{2}\sum_{i=1}^N C_{w_i,d_i}\sigma_{i}^{2}+\frac{N-1}{N}\sum_{i=1}^N s_{i}^{2}}{1-\left(\sigma_{\max}(\mathcal{M})\right)^{2}}.\label{bound_simp_trace}
$
Using Lemma~\ref{lem:M_properties} to expand $\left(\sigma_{\max}(\mathcal{M})\right)^{2}$ 
and then using $e_{ss}=\frac{d}{N}Tr(\cov_{\infty})$ completes the proof.\hfill$\qed$

 \bibliographystyle{elsarticle-num} 
 \bibliography{bibliography}
\end{document}